\documentclass[12pt,a4paper,twoside]{article} 
\usepackage[latin1]{inputenc}
\usepackage[matrix,arrow,ps,color,line,curve,frame,all]{xy} 
\SelectTips{cm}{}
\usepackage{hyperref}
\usepackage{amsfonts}
\usepackage{amssymb}
\usepackage{amsmath}
\usepackage{amsthm}
\usepackage[left=3cm,top=3cm,right=3cm,bottom=3cm]{geometry}
\usepackage{parskip}
\usepackage{verbatim}
\usepackage{dsfont}

\newtheoremstyle{standard}{10pt}{3pt}{\itshape}{}{\bfseries}{.}{.5em}{}
\theoremstyle{standard}
\newtheorem{lemma}{Lemma}[section]
\newtheorem{thm}[lemma]{Theorem}
\newtheorem{prop}[lemma]{Proposition}
\newtheorem{cor}[lemma]{Corollary}

\newtheoremstyle{definition}{10pt}{3pt}{}{}{\bfseries}{.}{.5em}{}
\theoremstyle{definition}
\newtheorem{defi}[lemma]{Definition}  
  
\newtheorem{expl}[lemma]{Example}  
\newtheorem{rem}[lemma]{Remark}

\DeclareMathOperator{\Spec}{Spec}

\DeclareMathOperator{\Q}{\mathbf{Qcoh}}

\DeclareMathOperator{\M}{\mathbf{Mod}}

\DeclareMathOperator{\supp}{supp}
\DeclareMathOperator{\End}{End}
\DeclareMathOperator{\id}{id}
\DeclareMathOperator{\Aut}{Aut}

\DeclareMathOperator{\Ann}{Ann}

\DeclareMathOperator{\coker}{coker}

\DeclareMathOperator{\Pic}{Pic}
\DeclareMathOperator{\Z}{Z}

\DeclareMathOperator{\Tor}{Tor}

\renewcommand{\O}{\mathcal{O}}
\newcommand{\U}{\mathcal{U}}

\newcommand{\A}{\mathcal{A}}

\newcommand{\T}{\mathcal{T}}

\renewcommand{\S}{\mathcal{S}}
\renewcommand{\L}{\mathcal{L}}

\newcommand{\p}{\mathfrak{p}}
\newcommand{\q}{\mathfrak{q}}

\title{Rosenberg's Reconstruction Theorem\\
(after Gabber)}
\author{Martin Brandenburg\footnote{WWU M\"unster, Fachbereich Mathematik und Informatik, \texttt{brandenburg@uni-muenster.de}}}
\date{\today}

\begin{document}

\maketitle

\begin{abstract}
\noindent Alexander L. Rosenberg has constructed a spectrum for abelian categories which is able to reconstruct a quasi-separated scheme from its category of quasi-coherent sheaves (\cite{R1}). In this note we present a detailed proof of this result which is due to Ofer Gabber.
\end{abstract}

\section{Introduction}
 
Our goal is to present a detailed and self-contained proof of the following Reconstruction Theorem:

\begin{thm}[Gabriel, Rosenberg]
Let $X,Y$ be quasi-separated schemes. If the categories $\Q(X)$ and $\Q(Y)$ are equivalent, then $X,Y$ are isomorphic.
\end{thm}
 
Actually we can classify all equivalences $\Q(X) \simeq  \Q(Y)$ in terms of isomorphisms $X \cong Y$ and line bundles on $Y$ (Theorem \ref{fine}). Note that quasi-separatedness is a very weak finiteness condition.
 
The idea is to associate to every abelian category $\A$ a ringed space $\Spec(\A)$ such that for every quasi-separated scheme $X$ we have an isomorphism of ringed spaces
\[X \cong \Spec(\Q(X)).\]
 
The first version of this Theorem for noetherian schemes was proved by Gabriel (\cite{G}) in 1962. In 1998 Rosenberg published a short proof for arbitrary schemes (\cite{R0}, \cite{R1}), then in 2004 a longer proof for quasi-separated schemes (\cite{R2}). Recently Antieau (\cite{A}) obtained a generalization of the Reconstruction Theorem to twisted quasi-coherent sheaves on quasi-compact quasi-separated schemes. With a completely new approach Calabrese and Groechenig have recently proven a Reconstruction Theorem for quasi-compact separated algebraic spaces (\cite{GC}).
 
Apart from being interesting in its own right, the Reconstruction Theorem is important for the vision of noncommutative algebraic geometry in which abelian categories are regarded as noncommutative schemes. For details we refer for example to the work by Rosenberg (\cite{R}), Artin and Zhang (\cite{AZ}).
 
Even more can be done if we use the tensor structure on quasi-coherent sheaves: Balmer has reconstructed a noetherian scheme from its tensor triangulated category of perfect complexes (\cite{B}). This has been generalized to quasi-compact quasi-separated schemes by Buan, Krause and Solberg (\cite{BKS}). Lurie has reconstructed an arbitrary geometric stack from its tensor category of quasi-coherent sheaves (\cite{L}). In the case of algebraic stacks the tensor structure becomes essential: The classifying stack $B C_2$ of a cyclic group of order $2$ over some nontrivial ring $R$ is not isomorphic to the scheme $\Spec(R) \sqcup \Spec(R)$, but they have equivalent categories of quasi-coherent sheaves provided that $2 \in R^*$.
 
The proof presented here grew out of an attempt to understand Rosenberg's proof. The one in (\cite{R1}) seems to be incomplete, the one in (\cite{R2}) is quite long and several arguments remained unclear to the author. We use Rosenberg's spectrum construction (with a slight but important modification due to Gabber, see Remark  \ref{differ}), which we will repeat here for the sake of completeness and the convenience of the reader. The proof that $\Spec(\Q(X))$ is isomorphic to $X$ for quasi-separated schemes $X$, including most of the necessary preparatory Lemmas in section \ref{asset}, is due to Gabber.

\emph{Acknowledgements.} I am indebted to Ofer Gabber for sharing his ingenious insights and making several suggestions for improvement. I would also like to thank Christopher Deninger and Alexandru Chirvasitu for helpful comments concerning the exposition, as well as Michael Groechenig for spotting some mistakes in the first version of this preprint.
 
\tableofcontents

\section{The spectrum as a set} \label{asset}

In the following $\A$ will always be an abelian category satisfying AB5, i.e. it is cocomplete and directed colimits are exact.
  
Recall that for objects $M,N \in \A$ we call $M$ a subquotient of $N$, when $M=P/Q$ for some subobjects $Q \subseteq P \subseteq N$. This is the case if and only if $M$ is a subobject of a quotient of $N$, if and only if $M$ is a quotient of a subobject of $N$. The relation of being a subquotient is reflexive and transitive. Besides, it is compatible with direct sums in the obvious sense.

\begin{defi}[Relation $\prec$ and spectral objects] \label{specs}
\noindent
\begin{enumerate}
\item Let $M,N \in \A$. We write $M \prec N$ when $M$ is a subquotient of a direct sum of (possibly infinitely many) copies of $N$. Note that $\prec$ is preserved by any cocontinuous exact functor. This fact will be used quite often.
\item We write $M \approx N$ and call $M,N$ \emph{equivalent}, when $M \prec N \prec M$. Clearly $\approx$ is an equivalence relation. 
\item For $M \in \A$ let $[M] := \{N \in \A : N \prec M\}$. Remark that we have $[M] \subseteq [N]$ if and only if $M \prec N$, and hence $[M]=[N]$ if and only if $M \approx N$.
\item We call $M \in \A$ \emph{spectral} if $M \neq 0$ and if for all subobjects $0 \neq N \subseteq M$ we have $M \prec N$ (and therefore $[M]=[N]$). Let $\Spec(\A)$ be the class of all $[M]$, where $M$ runs through the spectral objects of $\A$. Caution: The definition of a spectral object $M$ does not only depend on $[M]$.
\end{enumerate}
\end{defi}
 
\begin{rem} \label{differ}
Our definition of $\prec$ differs from Rosenberg's original definition which only involves finite direct sums and provides a reconstruction of \emph{quasi-compact} quasi-separated schemes. We will see in Lemma \ref{toreq} that infinite direct sums are useful.
\end{rem}
 
\begin{rem}[Size issues] \label{size1}
The definition of $\Spec(\A)$ causes set-theoretic difficulties, since each $[M]$ is a class and classes cannot be made up out of classes. There are several ways to remedy this.
\begin{itemize}
\item We can work with $\U$-categories for some universe $\U$ and realize $\Spec(\A)$ as a set in some larger universe. If $\A=\Q(X)$ for some $\U$-small quasi-separated scheme $X$, we will see later that in fact $\Spec(\A)$ can be identified with a $\U$-small set.
\item Assume that $\A$ is a category with a generator $P$. This is satisfied in most examples of interest. In particular, $\A$ is well-powered. Then it is easy to see that every object $M$ is equivalent to the direct sum of all $P/K$, where $K$ runs through a set of subobjects of $P$ such that $P/K$ admits an embedding into $M$. Hence, there is a set of representatives for $\approx$. In particular, $\Spec(\A)$ can be identified with a set.
\item The set-theoretic foundations are not essential for the proof of the Reconstruction Theorem. If $\Q(X) \cong \Q(Y)$ is an equivalence of categories, we will be able to produce an isomorphism $X \cong Y$ explicitly (see the proof of Theorem \ref{rekon}). In this note the spectrum is only an auxiliary construction.
\end{itemize}
\end{rem}
 
\begin{lemma} \label{supp}
Let $R$ be a commutative ring, $M \in \M(R)$ and $\p \in \Spec(R)$. Then $R/\p \prec M$ if and only if $M_{\p} \neq 0$.
\end{lemma}

\begin{proof}
If $R/\p \prec M$, then $0 \neq \mathrm{Frac}(R/\p) = (R/\p)_{\p} \prec M_{\p}$ implies $M_{\p} \neq 0$. Conversely, if $M_{\p} \neq 0$, choose some $m \in M$ such that $\Ann(m) \subseteq \p$. Then $R/\p$ is a quotient of $R/\Ann(m)$, and the latter admits a monomorphism to $M$ via multiplication with $m$. Hence $R/\p$ is a subquotient of $M$.
\end{proof}

\begin{prop} \label{specaff}
Let $R$ be a commutative ring. If $\p$ is a prime ideal of $R$, then $R/\p \in \M(R)$ is spectral. Every spectral object is equivalent to $R/\p$ for some prime ideal $\p$. For prime ideals $\p,\q$ we have $R/\p \prec R/\q$ resp. $R/\p \approx R/\q$ if and only if $\q \subseteq \p$ resp. $\p=\q$. Hence, there is a bijection
\[\Spec(R) \to \Spec(\M(R)),\, \p \mapsto [R/\p].\]
\end{prop}

\begin{proof}
Let $\p \in \Spec(R)$ and consider a submodule $0 \neq N \subseteq R/\p$. Choose some $0 \neq n \in N$. Multiplication with $n$ gives a monomorphism $R/\Ann(n) \to N$. But since $R/\p$ is an integral domain, we have $\Ann(n)=\p$. Hence, $R/\p \prec N$. This shows that $R/\p$ is spectral. Now let $M$ be spectral. Since all nontrivial subobjects of $M$ are equivalent to $M$, we may assume that $M$ is cyclic, and therefore $M=R/I$ for some proper ideal $I \subseteq R$. In order to show that $I$ is a prime ideal, let $r \in R \setminus I$, we have to show that $(I:r) \subseteq I$. Multiplication with $r$ gives a monomorphism $0 \neq R/(I:r) \to R/I$. Since $R/I$ is spectral, this implies $R/I \prec R/(I:r)$, and therefore $(I:r) \subseteq \Ann(R/I)=I$. The rest follows from Lemma \ref{supp} and $\supp(R/\p)=V(\p)$.
\end{proof}

In order to generalize this bijection to schemes, we will need some preparations. In the following let $X$ be a quasi-separated scheme (\cite[6.1]{EGAI}). This assumption will be used in the following way: For every open affine $j : U \hookrightarrow X$ the direct image functor $j_*$ preserves the property of being quasi-coherent (\cite[Proposition 6.7.1]{EGAI}).

\begin{defi}
For $M \in \Q(X)$ recall that $\Ann(M) := \ker(\O_X \to \End(M))$ is the annihilator ideal of $M$. If $M$ is of finite type, then $\Ann(M)$ is quasi-coherent (\cite[Proposition 2.2.2 (vi)]{EGAI}) and we have $V(\Ann(M))=\supp(M)$ as sets. If $X$ is integral with function field $K$, sheaf of meromorphic functions $\underline{K}$ and generic point $\eta : \Spec(K) \to X$, the torsion submodule $\Tor(M) \subseteq M$ is defined as the kernel of the canonical homomorphism $M \to \eta_* \eta^* M = M \otimes_{\O_X} \underline{K}$. It is quasi-coherent, and for $U \subseteq X$ open affine we have $\Gamma(U,\Tor(M))=\Tor(\Gamma(U,M))$. We call $M$ torsion-free if $\Tor(M)=0$. Observe that $M \approx N$ implies $\Ann(M)=\Ann(N)$ and $\supp(M)=\supp(N)$.
\end{defi}

\begin{lemma} \label{specsupp}
Let $M$ be spectral in $\Q(X)$. Then we have:
\begin{enumerate}
\item For all quasi-compact opens $U \subseteq X$ such that $M|_U \neq 0$ we have that $M|_U$ is spectral in $\Q(U)$.
\item $\Ann(M) \subseteq \O_X$ is quasi-coherent and as sets we have $V(\Ann(M))=\supp(M)$.
\item $\supp(M)$ is an irreducible closed subset of $X$.
\item The closed subscheme $Z:=V(\Ann(M))$ is an integral scheme.
\item If $j : Z \to X$ denotes the closed immersion, then $j^* M \in \Q(Z)$ is torsion-free.
\end{enumerate}
\end{lemma}

\begin{proof}
1. Let $0 \neq N \subseteq M|_U$ and $\overline{N} \subseteq M$ be the maximal quasi-coherent extension (\cite[Proposition 6.9.2]{EGAI}). Since $N \neq 0$ we have $\overline{N} \neq 0$. Hence $M \prec \overline{N}$ and then also $M|_U \prec \overline{N}|_U = N$.

2. For an open affine $U \subseteq X$ we have either $M|_U = 0$, hence $\supp(M|_U)=\emptyset$, or according to 1. and Proposition \ref{specaff} that $M|_U \approx \O_U / J_{x_U}$ for some unique point $x_U \in U$, where $J_x$ is the vanishing ideal of $\overline{\{x\}} \cap U$. Therefore $\supp(M|_U)$ equals $\overline{\{x\}} \cap U$. Because $\O_U/J_{x_U}$ is of finite type, $\Ann(M|_U)$ is quasi-coherent and we have $\Ann(M|_U)=\supp(M|_U)$ as sets. Since $X$ is covered by open affines, we are done.

3. With the above notation it suffices to prove that $x := x_U$ does not depend on $U$, because then $\supp(M)=\overline{\{x\}}$. First, let $V \subseteq U$ be an open affine such that $M|_V \neq 0$. The assumption $x_U \notin V$ implies $\overline{\{x_U\}} \cap V = \emptyset$ and therefore $(\O_U / J_{x_U})|_V = 0$, which is a contradiction. Hence $x_U \in V$ and using
\[\O_V / J_{x_V} \approx M|_V \approx (\O_U / J_{x_U})|_V = \O_V / J_{x_U}\]
we even get $x_U = x_V$. Now if $W \subseteq X$ is open affine with $M|_W \neq 0$, then we also have $M|_{U \cap W} \neq 0$: If not, we could find by gluing some $N \subseteq M|_{U \cup W}$ satisfying $N|_U = M|_U$ and $N|_W = 0$ and consider the maximal quasi-coherent extension $\overline{N} \subseteq M$. Because $M|_U \neq 0$ we have $\overline{N} \neq 0$. Because $M$ is spectral, we get $M \prec \overline{N}$ and then $M|_W \prec N|_W = 0$, a contradiction. Hence there is some open affine $V \subseteq U \cap W$ satisfying $M|_V \neq 0$. From what we already know we get $x_U = x_V = x_W$, as desired.

4. From 2. and 3. we infer that $Z$ is irreducible. The property of being reduced can be checked locally and follows since $J_{x_U}$ is a radical ideal of $\O_U$.

5. Because of $1$. we may assume that $X=\Spec(R)$ is affine. Then we already know that $M \approx R/\p$ for some prime ideal $\p$ and we have to show that $M \otimes_R R/\p$ is torsion-free over $R/\p$. Because of $\p = \Ann(M)$ we have $M \otimes_R R/\p = M$ as $R$-modules. Let $K$ be the torsion submodule of $M$ over $R/\p$, and assume $K \neq 0$. Considering $K$ as an $R$-module, since $M$ is spectral we have $M \prec K$ and therefore $M_\p \prec K_\p$. But then $M_\p \approx \mathrm{Frac}(R/\p)$ is torsion, a contradiction.
\end{proof}
 
\begin{lemma} \label{toreq}
Let $X$ be an integral scheme with function field $K$. Then every nontrivial torsion-free quasi-coherent module on $X$ is equivalent to $\underline{K}$.
\end{lemma}

\begin{proof}
Let $\eta : \Spec(K) \to X$ be the inclusion of the generic point and $M$ be a nontrivial torsion-free quasi-coherent module on $X$. Since $M$ embeds into $\eta_* \eta^* M$, which is a direct sum of copies of $\eta_* \widetilde{K} = \underline{K}$, we see $M \prec \underline{K}$. For the other direction, we have an epimorphism $\oplus_{f \in K} M \twoheadrightarrow M \otimes_{\O_X} \underline{K} = \eta_* \eta^* M$. Since $\eta_* \eta^* M$ is a nontrivial direct sum of copies of $\underline{K}$, it admits an epimorphism to $\underline{K}$. Hence, $\underline{K}$ is a quotient of $\oplus_{f \in K} M$, and therefore $\underline{K} \prec M$.
\end{proof}
 
\begin{prop} \label{rekon-menge}
Let $X$ be a quasi-separated scheme. Then the map
\[X \to \Spec(\Q(X)),\, x \mapsto [\O_X / J_x]\]
is a bijection. Here, $J_x$ denotes the vanishing ideal of $\overline{\{x\}}$. The inverse maps $[M]$ to the generic point of $\supp(M)$.
\end{prop}

\begin{proof}
Let us show that $\O_X/J_x$ is spectral. Consider a submodule $0 \neq M \subseteq \O_X / J_x$. Let $i : \overline{\{x\}} \to X$ denote the inclusion and endow $Z:= \overline{\{x\}}$ with the reduced subscheme structure, i.e. with the sheaf $\O_Z=i^{-1}(\O_X/J_x)$.
Because of $J_x M = 0$ we have $M \cong i_* i^* M$ and $0 \neq i^* M \subseteq \O_Z$. Since $i^* M$ and $\O_Z$ are nontrivial torsion-free quasi-coherent modules on the integral scheme $Z$, we have $\O_Z \prec i^* M$ by Lemma \ref{toreq}.
Applying $i_*$ we get $\O_X / J_x \prec M$, as desired. This shows that $X \to \Spec(\Q(X))$, $x \mapsto [\O_X / J_x]$ is well-defined. The map $\Spec(\Q(X)) \to X$ mapping $[M]$ to the generic point of $\supp(M)$ is well-defined because of Lemma \ref{specsupp}. The composite $X \to \Spec(\Q(X)) \to X$ is the identity since $\supp(\O_X/J_x)=\overline{\{x\}}$.

For the other composition, let $M$ be spectral in $\Q(X)$. With the notations of Lemma \ref{specsupp} we see that $j^* M$ is torsion-free on $Z=V(\Ann(M))=\overline{\{x\}}$. Since $M$ is annihilated by $\Ann(M)=J_x$, we have $M = j_* j^* M$, in particular $j^* M \neq 0$ and Lemma \ref{toreq} implies $j^* M \approx \O_Z$. Applying $j_*$ we get $M \approx \O_X / J_x$.
\end{proof}

\begin{defi}
For $M \in \A$ we define $\supp(M) := \{[P] \in \Spec(\A) : P \prec M\}$.
\end{defi}

\begin{lemma} \label{suppsch}
Let $M \in \Q(X)$ and $x \in X$. Then $\O_X / J_x \prec M$ if and only if $M_x \neq 0$. Hence, the bijection from Proposition \ref{rekon-menge} identifies the abstract support $\supp(M) \subseteq \Spec(\Q(X))$ with the usual support $\{x \in X : M_x \neq 0\} \subseteq X$.
\end{lemma}

\begin{proof}
We already know the affine case (Lemma \ref{supp}). Because of $(\O_X/J_x)_x \neq 0$ the one direction is clear. Now assume $M_x \neq 0$. Choose an open affine neighborhood $U$ of $x$, on which there is a local section $s$ of $M$ which does not vanish at $x$. Let $N \subseteq M|_U$ be the submodule generated by $s$ and $\overline{N} \subseteq M$ be the maximal quasi-coherent extension. We have $\overline{N}_x \neq 0$ and it suffices to prove that $\O_X / J_x \prec \overline{N}$. Thus, we may assume that $M|_U$ is generated by a single section $s$.

Consider $Z = \overline{\{x\}}$ as an integral closed subscheme of $X$ with closed immersion $i : Z \to X$. Let $N := (i^* M) / \Tor(i^* M)$. We claim that $\Gamma(i^{-1}(U),N) \neq 0$. In order to see this, we may replace $X$ by $U$ and therefore assume that $X=\Spec(R)$ is affine. Then $x$ corresponds to a prime ideal $\p$ and $M$ is associated to a cyclic $R$-module, say $M=R/I$ for some ideal $I$. Since $M_x \neq 0$ we have $I \subseteq \p$. Then $i^* M \cong R/I \otimes_R R/\p = R/\p$ and therefore $N \cong \O_Z \neq 0$.
 
Then $N \approx \O_Z$ by Lemma \ref{toreq} and therefore $i_* N \approx i_* \O_Z = \O_X / J_x$. Since $i_* N$ is a quotient of $i_* i^* M = M/J_x M$, we get $\O_X / J_x \prec M$ as desired.
\end{proof}

\section{The Zariski topology on the spectrum}

In the following we fix again an abelian category $\A$ satisfying AB5. By a subcategory we always mean a strictly full subcategory.

\begin{defi}
A subcategory $\T \subseteq \A$ is called \emph{topologizing} if $0 \in \T$ and $\T$ is closed under subobjects, quotients and arbitrary direct sums. In other words, $0 \in \T$ and $M \prec N \in \T$ implies $M \in \T$. In particular, $\T$ is itself a cocomplete abelian category. Note that for an object $M \in \A$ the smallest topologizing subcategory containing $M$ is given by $[M] = \{N \in \A : N \prec M\}$.
\end{defi}

\begin{defi}
A subcategory $\T \subseteq \A$ is called \emph{reflective} if the inclusion functor $\T \to \A$ has a left adjoint functor $F : \A \to \T$, called \emph{reflector}.
\end{defi}

The following two Lemmas are easy to prove:
 
\begin{lemma}
Let $\T \subseteq \A$ be a topologizing reflective subcategory and choose a reflector $F : \A \to \T$. Then the unit morphism $M \to F(M)$ is an epimorphism for all $M \in \A$, even an isomorphism in case of $M \in \T$.
\end{lemma}

\begin{lemma} \label{reflchar}
Let $\T \subseteq \A$ be a topologizing subcategory. Then $\T$ is reflective if and only if for every $M \in \A$ there is a smallest subobject $K \subseteq M$ such that $M/K \in \T$. In this case the reflector is given by $M \mapsto M/K$.
\end{lemma}
 
\begin{lemma}
Let $R$ be a commutative ring. There is an inclusion-reversing bijection between ideals of $R$ and topologizing reflective subcategories of $\M(R)$. Here, we map an ideal $I \subseteq R$ to $\T_I := \{M \in \M(R) : IM = 0\}$.
\end{lemma}

\begin{proof}
It is clear that $\T_I$ is topologizing. A reflector is given by $M \mapsto M/IM$. If $\T_J \subseteq \T_I$, we get $R/J \in \T_I$, which means $I \subseteq J$. The other direction is trivial. Now let $\T$ be a topologizing reflective subcategory with reflector $F : \M(R) \to \T$. Let $I$ be the kernel of the unit $R \to F(R)$, so that $F(R) \cong R/I$. We claim $\T = \T_I$. Since $I$ annihilates $F(R)$, the same is true for all $F(M)$, $M \in \M(R)$, since $F$ is cocontinuous. This implies $\T \subseteq \T_I$. Conversely, every object in $\T_I$ is a quotient of a direct sum of copies of $R/I \cong F(R) \in \T$, and therefore lies in $\T$.
\end{proof}

In the situation of this Lemma we have $R/\p \in \T_I$ if and only if $I \subseteq \p$, i.e. $\p \in V(I)$. Hence, under the bijection $\Spec(R) \cong \Spec(\M(R))$ from Proposition \ref{specaff} we see that $V(I) \subseteq \Spec(R)$ corresponds to the set of all $[M] \in \Spec(\M(R))$ with $M \in \T_I$. This motivates the following definition:

\begin{defi}
For a subcategory $\T \subseteq \A$ we define
\[V(\T) := \{[M] \in \Spec(\A) : M \in \T\}.\]
\end{defi}
In order to show the axioms for a topology, we need a replacement for the intersection and the product of ideals.

\begin{lemma} \label{intersec}
Let $\{\T_i\}_{i \in I}$ be a family of topologizing subcategories of $\A$. Then their intersection $\T = \cap_{i \in I} \T_i$ is also topologizing. If all $\T_i$ are reflective, the same is true for $\T$.
\end{lemma}

\begin{proof}
It is clear that $\T$ is topologizing. If $M \in \A$ and $K_i \subseteq M$ is the smallest subobject such that $M/K_i \in \T_i$, then $K:=\sum_{i \in I} K_i \subseteq M$ is the smallest subobject such that $M/K \in \T$.
\end{proof}
\begin{defi}
If $\S,\T$ are subcategories of $\A$, define the subcategory $\S \bullet \T$ of $\A$ as follows: We have $M \in \S \bullet \T$ if and only if there is an exact sequence
\[0 \to M' \to M \to M'' \to 0\]
such that $M' \in \T$ and $M'' \in \S$. This is called the \emph{Gabriel product} of $\S$ with $\T$. A mnemonic for the notation is ``$(\S \bullet \T)/\T = \S$''.
\end{defi}

\begin{expl}
In $\M(R)$ we have $\T_I \bullet \T_J = \T_{IJ}$ for ideals $I,J \subseteq R$: The inclusion $\subseteq$ is clear. For $\supseteq$ remark first that $R/(IJ) \in \T_I \bullet \T_J$ using the exact sequence
\[0 \to I/IJ \to R/(IJ) \to R/I \to 0\]
and then use the first part of the following Lemma.
\end{expl}

\begin{lemma} \label{gabprod}
Let $\S,\T$ be topologizing subcategories of $\A$.
\begin{enumerate}
\item Then $\S \bullet \T$ is topologizing.
\item If $\S,\T$ are reflective, the same is true for $\S \bullet \T$.
\end{enumerate}
\end{lemma}

\begin{proof}
1. It is clear that $0 \in \S  \bullet \T$. Now let $M \in \S \bullet \T$ and choose an exact sequence
\[0 \to M' \to M \to M'' \to 0\]
with $M' \in \T$ and $M'' \in \S$. If $Q \subseteq M$, then $Q':=Q \cap M' \subseteq M'$ satisfies $Q' \in \T$ and the image $Q'' \subseteq M''$ of $Q$ lies in $\S$. The sequence
\[0 \to Q' \to Q \to Q'' \to 0\]
is exact and shows $Q \in \S \bullet \T$. Now let $Q$ be a quotient of $M$. If $Q'$ denotes the image of $M'$ in $Q$, and $Q'' := Q/Q'$, then we have $Q' \in \T$ since it is a quotient of $M'$ as well as $Q'' \in \T$ because it is a quotient of $M''$. This shows $Q \in \S \bullet \T$. Since direct sums are exact in $\A$, we also see that $\S \bullet \T$ is closed under direct sums.

2. For $M \in \A$ let $K_{\S}(M)$ be the smallest subobject of $M$ such that $M/K_{\S}(M) \in \S$. Define $K_{\T}(M)$ analoguously. Let $K(M) := K_{\T}(K_{\S}(M))$. We claim that $K(M)$ is the smallest subobject of $M$ such that $M/K(M) \in \S \bullet \T$.

The chain of subobjects $K(M) \subseteq K_{\S}(M) \subseteq M$ gives an exact sequence
\[0 \to K_{\S}(M)/K(M) \to M/K(M) \to M/K_{\S}(M) \to 0.\]
Since $M/K_{\S}(M) \in \S$ and $K_{\S}(M)/K(M) \in \T$, we see $M/K(M) \in \S \bullet \T$. Now let $L \subseteq M$ be such that $M/L \in \S \bullet \T$. We want to prove $K(M) \subseteq L$. Choose an exact sequence
\[0 \to P \to M/L \to Q \to 0\]
with $P \in \T$ and $Q \in \S$. Write $P = M'/L$ with $L \subseteq M' \subseteq M$, so that $Q=M/M'$. Because of $Q \in \S$ we have $K_{\S}(M) \subseteq M'$. It follows
\[K_{\S}(M) / (L \cap K_{\S}(M)) \cong (K_{\S}(M)+L)/L \subseteq M'/L=P \in \T\]
and therefore $K(M) \subseteq L \cap K_{\S}(M) \subseteq L$.
\end{proof}

\begin{cor} \label{zartop}
The sets $V(-)$ enjoy the following properties:
\begin{enumerate}
\item We have $V(\{0\})=\emptyset$ and $V(\A)=\Spec(\A)$.
\item For a family of topologizing reflective subcategories $\{\T_i\}_i$ their intersection $\cap_i \T_i$ is a topologizing reflective subcategory such that $\cap_i V(\T_i) = V(\cap_i \T_i)$.
\item For two topologizing reflective subcategories $\S,\T$ the same is true for $\S \bullet \T$ and we have $V(\S) \cup V(\T) = V(\S \bullet \T)$.
\end{enumerate}
Hence, there is a topology on $\Spec(A)$, the \emph{Zariski topology}, in which the closed sets are those of the form $V(\T)$, where $\T \subseteq \A$ is a topologizing reflective subcategory.
\end{cor}

\begin{proof}
1. is clear and 2. follows from Lemma \ref{intersec}. 3. The inclusion $\subseteq$ follows from $\S \subseteq \S \bullet \T$ and $\T \subseteq \S \bullet \T$. For the other inclusion let $M \in \A$ be spectral with $M \in \S \bullet \T$. Choose an exact sequence $0 \to M' \to M \to M'' \to 0$ with $M' \in \T$ and $M'' \in \S$. If $M'=0$, we have $M \cong M'' \in \S$. Otherwise, $M \prec M'$ and therefore $M \in \T$.
\end{proof}

\begin{rem}[Size issues]
The proof of Lemma \ref{intersec} only works when $I$ is a $\U$-small set, when $\A$ is a $\U$-category. Since $\Spec(\A)$ does not lie in $\U$ in general, it may happen that arbitrary intersections of closed sets are not closed. This problem does not arise when $\A$ has a generator (see Remark \ref{size1}). Alternatively, if $\A$ is well-powered, then in the proof of the Lemma $\{K_i\}$ may be replaced by a $\U$-small set of subobjects.
\end{rem}
 
\begin{prop} \label{rekon-top}
Let $X$ be a quasi-separated scheme. Endow $\Spec(\Q(X))$ with the Zariski topology. Then the bijection (cf. Proposition \ref{rekon-menge})
\[X \to \Spec(\Q(X)),\, x \mapsto  [\O_X/J_x]\]
is a homeomorphism.
\end{prop}

\begin{proof}
The closed subsets of $X$ are the zero sets
\[V(I) = \supp(\O_X/I) = \{x \in X : I \subseteq J_x\}\]
of quasi-coherent ideals $I \subseteq \O_X$. The image of $V(I)$ under the bijection consists of all $[\O_X/J_x]$ such that $I \subseteq J_x$, i.e. $I (\O_X/J_x)=0$. Thus it equals $V(\T_I)$, where $\T_I := \{M \in \Q(X) : IM=0\}$ as in the affine case is topologizing and reflective. This shows that the bijection maps closed sets to closed sets. Conversely, let $\T$ be a topologizing reflective subcategory of $\Q(X)$. Let $I \subseteq \O_X$ be the smallest quasi-coherent ideal such that $\O_X/I \in \T$. Then $V(\T)$ consists of all $[\O_X/J_x]$ such that $\O_X / J_x \in \T$, i.e. $I \subseteq J_x$. Hence, $V(\T)$ is the image of $V(I)$.
\end{proof}

\subsubsection*{Classification of topologizing subcategories of $\Q(X)$}

We haven't classified yet the topologizing reflective subcategories of $\Q(X)$. In the proof of Proposition \ref{rekon-top} it was enough to consider their zero sets. As in the affine case every quasi-coherent ideal $I$ gives the topologizing reflective subcategory
\[\T_I := \{M \in \Q(X) : IM=0\}\]
such that $I \subseteq J$ if and only if $\T_J \subseteq \T_I$.
Rosenberg claims that these are all topologizing reflective subcategories (\cite[A1.2.6]{R2}) -- without proof. Ofer Gabber proved this under the assumption that $X$ is separated and communicated the proof to the author. We need some preparations first.

\begin{rem} \label{complete}
If $X$ is any scheme, choosing a big enough cardinal number $\kappa$, Gabber has proven (in a 1999 letter to Brian Conrad) that every quasi-coherent module on $X$ is the union of its $\kappa$-generated quasi-coherent submodules (\cite[Lemma 25.21.3]{SP}, \href{http://stacks.math.columbia.edu/tag/077K}{Tag 077K}).
This easily implies that $\Q(X)$ has a generator and is therefore a Grothendieck abelian category, in particular complete (loc. cit. Proposition 25.21.4). Note, however, that limits look quite different from the ones computed in the larger category of all $\O_X$-modules. For example, in the affine case $X=\Spec(A)$, the product of the quasi-coherent modules $\widetilde{M_i}$ is given by $\widetilde{\prod_i M_i}$. It admits a map to the product taken in the category of all $\O_X$-modules. On a basic open subset $D(f)$ it is given by the evident homomorphism $(\prod_i M_i)_f \to \prod_i (M_i)_f$, which is neither injective nor surjective in general unless the index set is finite.
\end{rem}

\begin{rem}
In the first published proof (\cite{EE}) of Gabber's result the equivalence between quasi-coherent modules and certain representations of quivers is not quite correct: A quasi-coherent module $M$ on $X$ is given by $\Gamma(U,\O_X)$-modules $M_U$ for every open affine $U \subseteq X$ and isomorphisms $\Gamma(V,\O_X) \otimes_{\Gamma(U,\O_X)} M_U \cong M_V$ for open affines $V \subseteq U \subseteq X$ which are compatible in the sense that the obvious diagram for $U \subseteq U$ and the obvious diagram for $U \subseteq V \subseteq W$ commutes. These compatibility conditions cannot be formulated in the language of quivers as in (\cite[Section 2]{EE}). Instead, one has to use small categories. The proof of Gabber's result (\cite[Corollary 3.5]{EE}) is easily generalized to this setting. 
\end{rem}

\begin{lemma} \label{reflprod}
Assume that $\A$ is also complete and well-powered. Let $\T \subseteq \A$ be a topologizing subcategory. Then, $\T$ is reflective if and only if $\T$ is closed under (possibly infinite) products taken in $\A$.
\end{lemma}

\begin{proof}
Assume that $\T$ is closed under products. Then $\T$ is closed under arbitrary limits, since these can be realized as subobjects of products. For $M \in \A$ let $L \to M$ be the intersection (i.e. pullback) of all subobjects $K \subseteq M$ such that $M/K \in \T$. Then $L \subseteq M$ and the monomorphism $M/L \to \prod_K M/K$ shows $M/L \in \T$. By construction $L$ is the smallest subobject of $M$ with this property.

If conversely $\T$ is reflective, consider a family of objects $\{M_i\}$ in $\T$ and choose the smallest subobject $L \subseteq \prod_i M_i$ with the property $(\prod_i M_i)/L \in \T$. The projection $p_i : \prod_i M_i \to M_i$ is a (split) epimorphism. It follows $L \subseteq \ker(p_i)$ for all $i$, hence $L=0$.
\end{proof}

\begin{lemma} \label{sep}
Let $X$ be a separated scheme and denote by $j : U \to X$ the inclusion of an open affine subset.
\begin{enumerate}
\item For every $M \in \Q(X)$ the canonical homomorphism
\[\Gamma(U,\O_X) \otimes_{\mathds{Z}} M \to j_* j^* M\]
is an epimorphism.
\item For every $N \in \Q(U)$ the direct image $j_* N$ is generated by global sections.
\end{enumerate}
\end{lemma}

\begin{proof} 1. On an open affine subset $V \subseteq X$ the homomorphism is given by
\[\Gamma(U,\O_X) \otimes_{\mathds{Z}} \Gamma(V,\O_X) \to \Gamma(U \cap V,\O_X),\]
which is surjective since $X$ is separated, tensored with $\Gamma(V,M)$ over $\Gamma(V,\O_X)$.

2. We have to prove that $\O_X^{\oplus \Gamma(X,j_* N)} \to j_* N$ is an epimorphism, i.e. that for every open affine $V \subseteq U$ the map $\Gamma(V,\O_X)^{\oplus \Gamma(U,N)} \to \Gamma(U \cap V,N)$ is an epimorphism. This factors as
\[\Gamma(V,\O_X)^{\oplus \Gamma(U,N)} \to \Gamma(V,\O_X) \otimes_{\mathds{Z}} \Gamma(U,N) \to \Gamma(U \cap V,N),\]
where the first map is obviously an epimorphism and also the second one by what we have seen in 1.
\end{proof}

\begin{prop}
Let $X$ be a separated scheme. Then $I \mapsto \T_I$ is an inclusion-reversing bijection between the quasi-coherent ideals of $\O_X$ and the topologizing reflective subcategories of $\Q(X)$.
\end{prop}

\begin{proof}
Let $\T \subseteq \Q(X)$ be a topologizing reflective subcategory. Let $I \subseteq \O_X$ be the smallest quasi-coherent ideal satisfying $\O_X/I \in \T$. We claim $\T = \T_I$.

Let $M \in \T$ and $s \in \Gamma(U,M)$ be a section on some open affine $j : U \hookrightarrow X$. Using Lemma \ref{sep} and choosing an epimorphism from a free abelian group to $\Gamma(U,\O_X)$, we see that $j_* j^* M$ is a quotient of a direct sum of copies of $M$, in particular it lies in $\T$. Identify $s$ with a homomorphism $s : \O_X \to j_* j^* M$. Its image lies in $\T$, so that its kernel contains $I$. This proves $IM=0$.
 
Conversely, let $M \in \Q(X)$ satisfy $IM=0$. Choose an open affine covering $\{u_i : U_i \to X\}$. Let $P$ be the product of the $(u_i)_* (u_i)^* M$ in $\M(X)$ and $Q$ the corresponding product in $\Q(X)$ (see Remark \ref{complete}). Then we have a commutative diagram
\[\xymatrix{M \ar[rr] \ar[dr] && P. \\ & Q \ar[ur] & }\]
Since $M \to P$ is a monomorphism, the same is true for $M \to Q$. We claim $M \in \T$. Because $\T$ is closed under subobjects and products (Lemma \ref{reflprod}), we only have to prove that $(u_i)_* (u_i)^* M \in \T$. But according to Lemma \ref{sep} this is generated by global sections which are annihilated by $I$, hence is a quotient of a direct sum of copies of $\O_X/I \in \T$ and therefore lies in $\T$.
\end{proof}

\section{The structure sheaf on the spectrum}

Recall that the \emph{center} of a category $C$ is the monoid $\Z(C)$ of all natural transformations $\id_C \to \id_C$. Thus, an element of $\Z(C)$ is a family of endomorphisms $\eta(M) : M \to M$ for every $M \in C$ such that for every morphism $f : M \to N$ in $C$ the diagram
\[\xymatrix{M \ar[r]^{\eta(M)} \ar[d]_f & M \ar[d]^f \\ N \ar[r]_{\eta(N)} & N}\]
commutes. Obviously $\Z(C)$ is a commutative monoid. If $C$ is linear, $\Z(C)$ has the structure of a commutative ring. For example, if $R$ is a ring, not assumed to be commutative, then an easy argument similar to the Yoneda Lemma shows that there is an isomorphism of rings $\Z(R) \cong \Z(\M(R))$, where $r \in R$ is mapped to the multiplication with $r$. In particular, if $R,S$ are commutative rings such that $\M(R)$, $\M(S)$ are equivalent (i.e. $R$, $S$ are \emph{Morita equivalent}), then $R$, $S$ are isomorphic. More generally, we have:

\begin{lemma} \label{modzen}
If $X$ is a ringed space and $\M(X)$ denotes the category of $\O_X$-modules, then $\Gamma(X,\O_X) \cong \Z(\M(X))$.
If $X$ is a quasi-separated scheme, then we also have $\Gamma(X,\O_X) \cong \Z(\Q(X))$. 
\end{lemma}

\begin{proof}
The homomorphism $\Gamma(X,\O_X) \to \Z(\M(X))$ maps a global section $s$ to the family of endomorphisms which multiply with $s$. The homomorphism in the other direction $\Z(\M(X)) \to \End(\O_X) \cong \Gamma(X,\O_X)$ is given by evaluation at $\O_X$. It is clear that $\Gamma(X,\O_X) \to \Z(\M(X)) \to \Gamma(X,\O_X)$ is the identity. It remains to prove that $\Z(\M(X)) \to \Gamma(X,\O_X)$ is injective.

So let $\eta \in \Z(\M(X))$ satisfy $\eta(\O_X)=0$. Let $M \in \M(X)$, we want to show that $\eta(M) \in \End(M)$ vanishes. On global sections this is clear, using naturality with respect to homomorphisms $\O_X \to M$.

Now let more generally $s \in \Gamma(U,M)$ be a local section. This corresponds to a global section $s'$ of $j_* j^* M$ such that $\Gamma(U,f)(s)=s'|_U$, where $f : M \to j_* j^* M$ denotes the canonical homomorphism. It follows
\[\Gamma(U,f)(\Gamma(U,\eta(M))(s))=\Gamma(U,\eta(j_* j^* M))(\Gamma(U,f)(s))=\Gamma(X,\eta(j_* j^* M))(s')|_U=0.\]
Since $f$ is an isomorphism on $U$, this means $\Gamma(U,\eta(M))(s)=0$.

If $X$ is a quasi-separated scheme, basically the same proof works: We may assume that $U$ is affine and therefore $j_*$ preserves quasi-coherence (\cite[Proposition 6.7.1]{EGAI}).
\end{proof}

Thus we can reconstruct global sections of $\O_X$ from $\Q(X)$. In order to get local sections, we need Gabriel's theory of quotients of abelian categories (\cite[Chapitre III]{G}). A subcategory $\T$ of an abelian category is called \emph{thick} if it contains $0$ and is closed under subquotients and extensions . Then we can construct the quotient $\A/\T$, which is an abelian category together with an exact functor $p : \A \to \T$ annihilating the objects of $\T$ and is universal with this property. Besides, one has the following properties:
\begin{enumerate}
\item The functor $p$ is a bijection on objects.
\item A morphism $f : M \to N$ becomes an isomorphism in $\A/\T$ if and only if $\ker(f) \in \T$ and $\coker(f) \in \T$.
\item Every short exact sequence in $\A/\T$ lifts to a short exact sequence in $\A$.
\end{enumerate}

\begin{expl}[\cite{G}, Example III.5.a] \label{lok}
Let $X$ be a ringed space and $U \subseteq X$ be an open subset. Then $\M_U(X) := \{M \in \M(X) : M|_U = 0\}$ is a thick subcategory of $\M(X)$ and the restriction functor $\M(X) \to \M(U)$ induces an equivalence of categories
\[\M(X) / \M_U(X) \simeq  \M(U).\]
If $X$ is quasi-separated, $U$ is affine and $\Q_U(X)$ is defined analogously, the same proof shows $\Q(X) / \Q_U(X) \simeq \Q(U)$.
\end{expl}

\begin{lemma} \label{zentrumlokal}
Let $\A$ be an abelian category and $\T \subseteq \A$ be a thick subcategory. Then $p : \A \to \A/\T$ induces a homomorphism of commutative rings $\Z(\A) \to \Z(\A/\T)$. These homomorphisms are compatible in the following sense:
If $\S \subseteq \A$ is another thick subcategory with $\T \subseteq \S$, then upon the identification $(\A/\T)/(\S/\T) \simeq \A/\S$ the following diagram commutes:
\[\xymatrix{\Z(\A) \ar[dr] \ar[rr] && \Z(\A/\S)  \\ &  \Z(\A/\T) \ar[ur] & }\]
\end{lemma}

\begin{proof}
Let $\eta \in \Z(\A)$. For $N \in \A/\T$ choose some $M \in \A$ with $p(M)=N$ and define $\eta|_\T(N) = p(\eta(M)) \in \End(N)$. We claim $\eta|_T \in \Z(\A/\T)$. If $f : N \to N'$ is a morphism in $\A/\T$, we have to prove that
\[\xymatrix{N \ar[r]^{f} \ar[d]_{\eta|_T(N)} & N' \ar[d]^{\eta|_T(N')} \\ N \ar[r]^f & N'}\]
commutes, which a priori is only clear if $f$ is induced by a morphism $M \to M'$. In the general case there are (cf. \cite[III.1, Proof of Proposition 1]{G}) morphisms $a : X \to M$, $b : M' \to X'$, $g : X \to X'$ in $\A$ which fit into a commutative diagram
\[\xymatrix{N \ar[r]^{f} & N'  \ar[d]^{p(b)} \\ p(X) \ar[u]^{p(a)} \ar[r]_{p(g)} & p(X'),}\]
in which $p(a)$ and $p(b)$ are isomorphisms. This yields the commutative diagram
\[\xymatrix{ p(X) \ar[rr]^<<<<<<<<<<<<<{p(g)} \ar[dr]_-{p(a)} \ar[dd]_>>>>>>>>>>>>>>>>>{p(\eta(X))} && p(X') \ar'[d]_>>>>>{p(\eta(X'))} [dd] & \\
& N \ar[rr]_<<<<<<<<<<{f} \ar[dd]_<<<<<<{\eta|_T(N)} && N' \ar[dd]_<<<<<<{\eta|_T(N')} \ar[ul]_-{p(b)} \\
p(X) \ar'[r][rr]^<<<<{p(g)} \ar[dr]_{p(a)}  && p(X')
\\& N \ar[rr]_<<<<<<<<<<{f} && N' \ar[ul]_{p(b)}. }\]
The left, back and right face commute because of $\eta \in \Z(\A)$. The top face equals the bottom face and is nothing else than the previous commutative diagram. Because  $p(a)$ and $p(b)$ are isomorphisms, the front face also commutes, as desired. This finishes the proof that $\eta|_\T \in \Z(\A/\T)$. The rest is obvious.
\end{proof}

\begin{rem} \label{comp}
If $U$ is an open subset of a ringed space $X$, then the diagram
\[\xymatrix@C=20pt{\Z(\M(X)) \ar[r]^-{\ref{zentrumlokal}} \ar[d]_-{\ref{modzen}}^{\cong} & \Z(\M(X)/\M_U(X)) \ar[r]_-{\cong}^-{\ref{lok}} & \Z(\M(U)) \ar[d]^{\ref{modzen}}_{\cong} \\
\Gamma(X,\O_X) \ar[rr]^{\mathrm{res}} & &  \Gamma(U,\O_X)}\]
commutes. Again, the same is true for quasi-coherent modules on a quasi-separated scheme $X$ when $U$ is affine.
\end{rem}

In the following $\A$ is again an abelian category satisfying AB5.

\begin{defi}
For $[P] \in \Spec(\A)$ let $\langle [P] \rangle := \{M \in \A : P \not\prec M\}$. For a subset $U \subseteq \Spec(\A)$ let
\[\langle U \rangle := \bigcap_{[P] \in U} \langle [P] \rangle.\]
\end{defi}

\begin{lemma} \label{isthick}
Let $U \subseteq \Spec(\A)$. Then $\langle U \rangle$ is a thick subcategory of $\A$.
\end{lemma}

\begin{proof}
It suffices to show that $\langle [P] \rangle$ is thick for spectral objects $P \in \A$. It is clear that it is topologizing. It remains to show that in an exact sequence
\[0 \to M' \to M \to M'' \to 0\]
with $P \prec M$, we have $P \prec M'$ or $P \prec M''$.
Using the same arguments as in the first part of Lemma \ref{gabprod}, we may assume $M=P$. But then we may repeat the argument of the third part of Corollary \ref{zartop} to obtain $P \prec M'$ or $P \cong M''$.
\end{proof}

\begin{defi}
Endow $X = \Spec(\A)$ with the Zariski topology. For $U \subseteq X$ open thanks to Lemma \ref{isthick} the quotient $A/\langle U \rangle$ makes sense and we define
\[\O'_X(U) := \Z(\A/\langle U \rangle).\]
For $U \subseteq V$ we have $\langle V \rangle \subseteq \langle U \rangle$ and the projection $\A / \langle V \rangle \to \A/\langle U \rangle$ induces a natural homomorphism $\O'_X(V) \to \O'_X(U)$ (Lemma \ref{zentrumlokal}). Then $\O'_X$ is a presheaf of commutative rings on $X$. Let $\O_X$ be the sheaf associated to $\O'_X$. Then $(X,\O_X)$ is a ringed space, the \emph{spectrum} of $\A$, which is also denoted by $\Spec(\A)$.
\end{defi}

\begin{prop} \label{rekon-sch}
Let $X$ be a quasi-separated scheme. Then the homeomorphism from Proposition \ref{rekon-top} extends to an isomorphism of ringed spaces
\[X \cong \Spec(\Q(X)).\]
\end{prop}

\begin{proof}
If $U' \subseteq X$ is an open affine, the image $U \subseteq \Spec(\Q(X))$ consists of all $[\O_X/J_x]$ with $x \in U'$. From Lemma \ref{suppsch} we get $\langle U \rangle = \Q_{U'}(X)$. Hence, Lemma \ref{modzen} and Example \ref{lok} yield
\[\O'_{\Spec(\Q(X))}(U) = \Z(\Q(X)/\Q_{U'}(X)) \cong \Z(\Q(U')) \cong \Gamma(U',\O_X).\]
By Lemma \ref{zentrumlokal} and Remark \ref{comp} these isomorphisms are compatible with respect to restrictions of open affine subsets.
\end{proof}

\section{Proof of the Theorem and variants}
 
\begin{thm}[Reconstruction Theorem] \label{rekon}
Let $X,Y$ be quasi-separated schemes. If the categories $\Q(X)$ and $\Q(Y)$ are equivalent, then $X,Y$ are isomorphic.
\end{thm}

\begin{proof}
It is clear that an equivalence of abelian categories $\A \simeq \A'$ induces an isomorphism of ringed spaces $\Spec(\A) \cong \Spec(\A')$. Hence, using Proposition \ref{rekon-sch}, an equivalence $F : \Q(X) \simeq \Q(Y)$ induces an isomorphism of ringed spaces
\[f : X \cong \Spec(\Q(X)) \cong \Spec(\Q(Y)) \cong Y.\]
Explicitly, for $x \in X$ we have $F(\O_X / J_x) \approx \O_Y / J_{f(x)}$, and for some open subset $V \subseteq Y$ the composition
\[\Q(V) \simeq \Q(Y)/\Q_V(Y) \simeq \Q(X)/\Q_{f^{-1}(V)}(Y) \simeq \Q(f^{-1}(V))\]
induces a ring isomorphism $f^\#(V) : \O_Y(V) \to \O_X(f^{-1}(V))$ on centers.
\end{proof}
 
\begin{rem}
Actually the same proof works over any commutative ground ring $R$. If $X,Y$ are quasi-separated $R$-schemes such that $\Q(X)$ and $\Q(Y)$ are equivalent as $R$-linear categories, then $X,Y$ are isomorphic $R$-schemes. 
\end{rem}

\begin{rem}
The Reconstruction Theorem does not say that every equivalence $\Q(X) \simeq \Q(Y)$ is isomorphic to $f^*$ for some isomorphism $f : Y \cong X$. In fact, pullback functors preserve the structure sheaf and tensoring with a line bundle on $X$ induces an equivalence $\Q(X) \simeq \Q(X)$. But this is the only obstruction, as the following refinement of the Reconstruction Theorem shows. 
\end{rem}

\begin{thm} \label{fine}
Let $X,Y$ be quasi-separated schemes. Then there is an equivalence of groupoids
\[\{(f : Y \cong X,~\L \text{ line bundle on } Y)\} \simeq \{\text{equivalences } \Q(X) \simeq \Q(Y)\}\]
given by mapping $(f,\L)$ to $f^*(-) \otimes \L$.
\end{thm}

A morphism $(f,\L) \to (f',\L')$ exists if and only if $f=f'$ and is given by an isomorphism $\L \cong  \L'$. It is mapped to the  isomorphism $f^*(-) \otimes \L \cong f^*(-) \otimes \L'$ of functors which tensors with $\L \cong \L'$.
 
\begin{proof}
The functor is clearly faithful. For fullness, given $f^*(-) \otimes \L \cong g^*(-) \otimes \L'$, then evaluation at $\O_X$ gives $\L \cong \L'$, and we get an isomorphism $\alpha : f^* \cong g^*$. We have to show that $f=g$ and that $\alpha$ multiplies $\alpha(\O_X)(1) \in \Gamma(Y,\O_Y)^*$. We may work locally on $Y$, so assume that $Y$ and hence $X$ is affine. If $f^\#,g^\# : R \to S$ are the corresponding ring homomorphisms, then by the Eilenberg-Watts Theorem (\cite{W}) $\alpha$ corresponds to an isomorphism of $(R,S)$-bimodules $h : (f^\#,S,\id_S) \to (g^\#,S,\id_S)$. Then $h(s)=h(1)s$ for every $s \in S$, and $h(1) f^\#(r) = h(f^\#(r)) = g^\#(r) h(1)$,  so that $f^\#(r)=g^\#(r)$ for every $r \in R$.
 
For essential surjectivity, let $F : \Q(X) \to \Q(Y)$ be an equivalence of categories. Let $\L := F(\O_X)$. By Proposition \ref{rekon-sch}, $F$ induces an isomorphism $f : Y \to X$ characterized by the properties that $F(\O_X / J_{f(y)}) \approx \O_Y / J_y$ for all $y \in Y$ and that for every open affine $U \subseteq Y$, the composition
\[F_U : \Q(f(U)) \simeq \Q(X) / \Q_{f(U)}(X) \simeq \Q(Y) / \Q_U(Y)  \simeq \Q(U)\]
induces $f^\# : \O_X(f(U)) \to \O_Y(U)$ on centers. It follows that the composition
\[\Q(U) \xrightarrow{(f|_U^{-1})^*} \Q(f(U)) \xrightarrow{F_U} \Q(U)\]
induces the identity on centers. By the Eilenberg-Watts Theorem this functor is given by tensoring with some quasi-coherent $\O_U$-bimodule, namely $\L|_U$, but the two actions from $\O_U$ coincide. Hence $F_U$ is isomorphic to $(f|_U)^*(-) \otimes \L|_U$. These isomorphisms glue to an isomorphism $F \cong f^*(-) \otimes \L$. Since $F$ is an equivalence, $\L$ is invertible, hence a line bundle.
\end{proof}

See \cite[Theorem 4.2]{GC} for the corresponding result for quasi-compact separated flat algebraic spaces. Recall that the automorphism class group of a category consists of its isomorphism classes of auto-equivalences, endowed with the obvious group structure (\cite{CB}).

\begin{cor}
If $X$ is a quasi-separated scheme, then the automorphism class group of $\Q(X)$ is isomorphic to the semidirect product $\Aut(X) \ltimes \Pic(X)$, where $\Aut(X)$ acts on $\Pic(X)$ via direct images.
\end{cor}
 
The following variant takes into account the tensor structure and doesn't restrict to isomorphisms resp. equivalences:
 
\begin{thm}[\cite{BC}, Theorem 3.4.3; \cite{S1}, Theorem 1.3.2]
If $X$ is a quasi-compact quasi-separated scheme, and $Y$ is an arbitrary scheme, then $f \mapsto f^*$ induces an equivalence of categories between the (discrete) category of morphisms $Y \to X$ and the category of cocontinuous symmetric monoidal functors $\Q(X) \to \Q(Y)$. The same holds when $X$ is an Adams stack, i.e. a geometric stack with enough locally free sheaves.
\end{thm}
 
This could be seen as a functorial Reconstruction Theorem. In fact, whereas Rosenberg's spectrum is only functorial with respect to equivalences of abelian categories, the spectrum in (\cite{BC}, Definition 3.5.1) is functorial with respect to all cocontinuous symmetric monoidal functors and the Theorem above comes down to $\Spec\bigl(\Q(X)\bigr) \cong X$. It would be interesting to extend this to all geometric stacks $X$, and to find an intrinsic characterization of those cocomplete symmetric monoidal categories of the form $\Q(X)$. See (\cite{S2}) for the special case of Adams stacks.


\bigskip 

 
\begin{thebibliography}{boint}   
\bibitem[A]{A} B. Antieau, \emph{Caldararu's conjecture on abelian categories of twisted sheaves}, preprint, arXiv:1305.2541
\bibitem[AZ]{AZ} M. Artin, J. J. Zhang, \emph{Noncommutative projective schemes}, Adv. Math. 109 (1994), no. 2, 228-287
\bibitem[B]{B} P. Balmer, \emph{Presheaves of triangulated categories and reconstruction of schemes}, Math. Ann. 324 (2002), no. 3, 557-580.
\bibitem[BC]{BC} M. Brandenburg, A. Chirvasitu. \emph{Tensor functors between categories of quasi-coherent sheaves}, Journal of Algebra 399 (2014): 675-692. arXiv:1202.5147
\bibitem[CB]{CB} W. E. Clark, G. M. Bergman, \emph{The Automorphism Class Group of the Category of Rings}, J. Algebra 24 (1973), 80-99
\bibitem[BKS]{BKS} A. B. Buan, H. Krause, O. Solberg, \emph{Support varieties: an ideal approach}, Homology, Homotopy Appl. 9 (2007), no. 1, 45-74.
\bibitem[EE]{EE} E. Enochs, S. Estrada, \textit{Relative homological algebra in the category of quasi-coherent sheaves}, Adv. Math. 194, 2005, 284-295.
\bibitem[G]{G} P. Gabriel, \textit{Des cat\'{e}gories ab\'{e}liennes}, Bull. Soc. Math. France 90, 1962, 323-448.
\bibitem[GC]{GC} M. Groechenig, J. Calabrese, \emph{Gabriel's Theorem and Moduli of points}, preprint, arXiv:1310.6600
\bibitem[GD]{EGAI} A. Grothendieck, J. Dieudonn\'{e}, \textit{\'{E}l\'{e}ments de g\'{e}om\'{e}trie alg\'{e}brique (r\'{e}dig\'{e}s avec la collaboration de Jean Dieudonn\'{e}) : I. Le langage des sch\'{e}mas}, Grundlehren der Mathematischen Wissenschaften 166, Springer-Verlag, 2nd edition, 1971.
\bibitem[L]{L} J. Lurie, \textit{Tannaka Duality for Geometric Stacks}, preprint,\\arXiv:math/0412266, 2005.
\bibitem[R]{R} A. L. Rosenberg, \emph{Noncommutative algebraic geometry and representations of quantized algebras}, MIA 330, Kluwer Academic Publishers Group, Dordrecht, 1995.
\bibitem[R0]{R0} A. L. Rosenberg, \textit{Reconstruction of Schemes}, MPI Preprints Series, 1996 (108).
\bibitem[R1]{R1} A. L. Rosenberg, \emph{The spectrum of abelian categories and reconstruction of schemes}, Rings, Hopf algebras, and Brauer groups, Lect. Notes Pure Appl. Math., vol. 197, Marcel Dekker, New York, 1998,  257-274.
\bibitem[R2]{R2} A. L. Rosenberg, \textit{Spectra of 'spaces' represented by abelian categories}, MPI Preprints Series, 2004 (115).
\bibitem[S1]{S1} D. Sch\"{a}ppi, \emph{A characterization of categories of coherent sheaves of certain algebraic stacks}, Preprint, arXiv:1206.2764
\bibitem[S2]{S2} D. Sch\"{a}ppi, \emph{Which abelian tensor categories are geometric?}, Preprint, arXiv:1312.6358
\bibitem[SP]{SP} Stacks Project, Open source textbook on algebraic stacks and the algebraic geometry needed to define them, \url{http://stacks.math.columbia.edu}
\bibitem[W]{W} C. E. Watts, \emph{Intrinsic characterizations of some additive functors}, Proc. Amer. Math. Soc. 11 (1960), 5-8
\end{thebibliography}
\end{document}